\documentclass{article}
\usepackage{amssymb}

\usepackage{amsmath}

\newtheorem{theorem}{Theorem}[section]

\newtheorem{algorithm}[theorem]{Algorithm}

\newtheorem{corollary}[theorem]{Corollary}

\newtheorem{definition}[theorem]{Definition}
\newtheorem{example}[theorem]{Example}

\newtheorem{lemma}[theorem]{Lemma}

\newtheorem{proposition}[theorem]{Proposition}
\newtheorem{remark}[theorem]{Remark}

\newtheorem{Main Result:}{Main Result:}
\newenvironment{proof}[1][Proof]{\textbf{#1.} }{\ \rule{0.5em}{0.5em}} %
\begin{document}

\title{{Finding normal bases over finite fields with prescribed trace self-orthogonal relations}
\author{Xiyong Zhang\footnote{\textbf{Corresponding E-mail Address:} xiyong.zhang@hotmail.com}, \ \
Rongquan Feng,\ \ Qunying Liao, \ \ Xuhong Gao  \\
}
\date{}
}
\maketitle

\vspace{-0.4cm}

\begin{abstract}
Normal bases and self-dual normal bases over finite fields have been
found to be very useful in many fast arithmetic computations. It is
well-known that there exists a self-dual normal basis of
$\mathbb{F}_{2^n}$ over $\mathbb{F}_2$ if and only if $4\nmid n$. In
this paper, we prove there exists a normal element $\alpha$ of
$\mathbb{F}_{2^n}$ over $\mathbb{F}_{2}$ corresponding to a
prescribed vector $a=(a_0,a_1,\cdots,a_{n-1})\in \mathbb{F}_2^n$
such that $a_i=\mbox{Tr}_{2^n|2}(\alpha^{1+2^i})$ for $0\leq i\leq
n-1$, where $n$ is a 2-power or odd, if and only if the given vector
$a$ is symmetric ($a_i=a_{n-i}$ for all $i, 1\leq i\leq n-1$), and
one of the following is true.

1) $n=2^s\geq 4$, $a_0=1$, $a_{n/2}=0$, $\sum\limits_{1\leq i\leq
n/2-1, (i,2)=1}a_i=1$;

2) $n$ is odd, $(\sum\limits_{0\leq i\leq n-1}a_ix^i,x^n-1)=1$.

Furthermore we give an algorithm to obtain  normal elements
corresponding to prescribed vectors in the above two cases. For a
general positive integer $n$ with $4|n$, some necessary conditions
for a vector to be the corresponding vector of a normal element of
$\mathbb{F}_{2^n}$ over $\mathbb{F}_{2}$ are given. And for all $n$
with $4|n$, we prove that there exists a normal element of
$\mathbb{F}_{2^n}$ over $\mathbb{F}_2$ such that the Hamming weight
of its corresponding vector is 3, which is the lowest possible
Hamming weight.

\end{abstract}

\par \textbf{Keywords:}
{\textit{Normal basis, Self-dual, Hamming weight, Reciprocal
polynomial, trace function}

\section{Introduction}

Suppose $\mathbb{F}_{q^n}$ is the extension field of degree $n$ over
the $q$ elements finite field $\mathbb{F}_q$, where $q$ is a prime
power. The trace function $\mathbb{F}_{q^n}$ to its subfield
$\mathbb{F}_{q^m}$ with $m|n$ is defined by
$\mbox{Tr}_{q^n|q^m}(x)=\sum\limits_{0\leq i\leq n/m-1}x^{q^{im}}$
for $x\in \mathbb{F}_{q^n}$.

A \textbf{normal basis} of $\mathbb{F}_{q^n}$ over $\mathbb{F}_{q}$ is a
basis of the form $N =\{\alpha,\alpha^q,\cdots, \alpha^{q^{n-1}}\}$,
where $\alpha$ is called a \textbf{normal element} of
$\mathbb{F}_{q^n}$ over $\mathbb{F}_{q}$. It is well-known that
there exists a normal basis for every finite field extension
$\mathbb{F}_{q^n}$ over $\mathbb{F}_{q}$.

Let $\{\alpha,\alpha^q, \cdots, \alpha^{q^{n-1}}\}$ and
$\{\beta,\beta^q,\cdots, \beta^{q^{n-1}}\}$ are two normal bases
of $\mathbb{F}_{q^n}$ over $\mathbb{F}_{q}$, if
\begin{equation*}
\mbox{Tr}_{q^n|q}(\alpha^{q^i}\cdot \beta^{q^j})=\left\{
\begin{array}{ll}
1,\ & i=j,\\
0,\ & i\neq j,
\end{array}
\right.
\end{equation*}
then $\{\beta,\beta^q,\cdots, \beta^{q^{n-1}}\}$ is called the
\textbf{dual basis} of $\{\alpha,\alpha^q,\cdots, \alpha^{q^{n-1}}\}$. If
$\alpha=\beta$, then the normal basis $\{\alpha,\alpha^q,\cdots,
\alpha^{q^{n-1}}\}$ is called \textbf{self-dual}.

It is well-known that normal bases and self-dual normal bases over finite
fields are very useful for some fast arithmetic computations (For example efficient
exponentiation, as $q$-th power of an element  is only  a cyclic
bit-shift of its coordinate vector). They are used to design simple and fast multipliers of
finite fields, which is one of the most time-consuming operations. Massey and Omura \cite{MO86}
patented a hardware multiplier under normal bases over
$\mathbb{F}_2$ and modifications of this multiplier  can be found in
\cite{wang,HW,RH,RH05,SK01}. Normal bases have also been implemented
efficiently in software, see, for example, \cite{DH,G00,VAZ}. Generally, it is known that normal bases and self-dual normal bases
have wide applications such as in coding theory, cryptography,
signal processing, etc.

Constructions of normal bases and self-dual bases have been
extensively studied in the past two decades. A non exhaustive list
of references is \cite{scheer,poli,nogami,pickett,arnau}.  The latest
results can be found for instance in \cite{pickett} and \cite{arnau}, where
explicit constructions of self-dual (integral) normal bases in
abelian extensions of finite and local fields were given.

The followings are standard results about normal bases.
\begin{theorem}\label{nbth}
(The normal basis theorem) For any prime power $q$ and positive
integer $n$, there is a normal basis in $\mathbb{F}_{q^n}$ over
$\mathbb{F}_q$.
\end{theorem}

\begin{theorem}\cite{lempel}\label{scn}
There is a self-dual normal basis of $\mathbb{F}_{q^n}$ over
$\mathbb{F}_{q}$ if and only if one of the following is true.

1) $q$ and $n$ are odd;

2) $q$ is even and $n\neq 0\pmod{4}$.
\end{theorem}

A vector $(a_0,a_1,\cdots,a_{n-1})$ in $\mathbb{F}_q^n$ is called to
be \textbf{symmetric} if $a_i=a_{n-i}$ for all $i=1,
\ldots, n-1$, and a polynomial
$f(x)=a_0+a_1x+a_2x^2+\cdots+a_{n-1}x^{n-1}$ in
$\mathbb{F}_q[x]/(x^n-1)$ is \textbf{symmetric} if
$a_i=a_{n-i}$ for all $i, 1\leq i\leq n-1$.

For $\alpha\in \mathbb{F}_{q^n}$, let
$a_i=\mbox{Tr}_{q^n|q}(\alpha\cdot \alpha^{q^i})$, $0\leq i\leq
n-1$. Throughout the paper, the vector
$a=(a_0,a_1,\cdots,a_{n-1})\in \mathbb{F}_{q}^n$ is called
\textbf{the corresponding vector} of $\alpha$, and the vector $a$
can be viewed as \textbf{the trace self-orthogonal relation} of
$\alpha$. It is obvious that $a$ and $f(x)=\sum\limits_{0\leq i\leq
n-1}a_ix^i$ are both symmetric.

Our interests in normal bases with good
trace self-orthogonal relations stem both from mathematical theory and
practical applications, here we mean by \textbf{good} that the corresponding
vector of a normal element has the lowest possible Hamming weight. For example, by using normal bases with good
trace self-orthogonal relations, one can build a more simpler
corresponding relation between trace functions on finite field
$\mathbb{F}_{2^n}$ and boolean functions on $\mathbb{F}_2^n$.
When a normal basis is used, the function $f(x)=Tr_{2^n|2}(x^d)\in \mathbb{F}_{2^n}[x], 1<d<2^n-1$, can be transformed to
the so-called \textit{rotation symmetric boolean function} over $\mathbb{F}_2^n$, where \textit{rotation symmetric boolean functions}
have been found to have important applications in the design of cryptographic algorithms \cite{kavut}. Let $\alpha\in\mathbb{F}_{2^n}$ be a normal element,
then $x=x_0\alpha+x_1\alpha^2+\cdots+x_{n-1}\alpha^{2^{n-1}}$ for all $x\in\mathbb{F}_{2^n}$,
 where $(x_0,\cdots,x_{n-1})\in \mathbb{F}_2^n$. In the case that $d=2^i+1(0<i<n)$,
\begin{equation*}
\begin{array}{ll}
&Tr_{2^n|2}(x^{1+2^i})\\
&=Tr_{2^n|2}((x_0\alpha+x_1\alpha^2+\cdots+x_{n-1}\alpha^{2^{n-1}})(x_0\alpha^{2^i}+x_1\alpha^{2^{i+1}}+\cdots+x_{n-1}\alpha^{2^{i-1}}))\\
&=\sum\limits_{1\leq j<\lceil n/2\rceil}Tr_{2^n|2}(\alpha\alpha^{2^{i+j}}+\alpha\alpha^{2^{i-j}})(x_0x_j+x_1x_{j+1}+\cdots+x_{n-1}x_{j-1})\\
&\ \ \ \ \ \ \ \ \ \ \ +Tr_{2^n|2}(\alpha\alpha^{2^i})(x_0+\cdots+x_{n-1})\\
&=f_\alpha(x_0,\cdots,x_{n-1}).
\end{array}
\end{equation*}
Thus if the corresponding vector of $\alpha$ has fewer $1$s, the rotation symmetric function $f_\alpha(x_0,\cdots,x_{n-1})$  has fewer cycles.

At the practical aspect, by using normal bases with good trace
self-orthogonal relations, one can achieve high computation
efficiency in the implementations of finite field arithmetic in some
cryptography systems or communication systems. For example a self-dual normal basis multiplier
was presented by Wang \cite{wang} with very low complexity. Whenever $\mathbb{F}_{2^n}$ doesn't has a self-dual normal basis over $\mathbb{F}_{2}$, we can similarly design multiplier by using normal basis with good trace
self-orthogonal relations to reduce the number of trace $Tr_{2^n|2}(\alpha^{1+2^i+2^j})$ computations.

While by Theorem \ref{scn}, there doesn't exist self-dual normal
basis of $\mathbb{F}_{2^n}$ over $\mathbb{F}_{2}$ when $4|n$.
Therefore for any $n$ with $4|n$, one can ask the following
questions naturally.

(1)\quad What the trace self-orthogonal relation of a normal basis
of $\mathbb{F}_{2^n}$ over $\mathbb{F}_{2}$ should be?

(2)\quad  What is the lowest possible Hamming weight of the vector
corresponding to a normal basis? Or generally what does a valid
vector in $\mathbb{F}_2^n$ corresponding to a normal element look
like?

(3)\quad For a given valid vector in $\mathbb{F}_2^n$, can one
construct a normal element of $\mathbb{F}_{2^n}$ over
$\mathbb{F}_{2}$ corresponding to such a given vector?

To our knowledge, there have not been such researches. In the present paper, we focus our attention
on the above questions. Characterizations of special differential
factorizations of some polynomials will be given, by which we obtain
the necessary and sufficient conditions for a vector corresponding
to a normal element of $\mathbb{F}_{2^n}$ over $\mathbb{F}_{2}$
where $n$ is a 2-power or an odd number. Also we present an algorithm
to find normal elements corresponding to a prescribed vector, and
show that there exists a normal element of $\mathbb{F}_{2^n}$ over
$\mathbb{F}_2$ for every $n$ with $4|n$ such that the Hamming weight
of its corresponding vector is 3, which is the lowest possible
Hamming weight.

For our method to work, we define the generalized reciprocal
polynomial of a polynomial in $\mathbb{F}_{2^m}[z]/(z^n-1)$ as following.
\begin{definition}
Let $g(z)=\sum\limits_{0\leq i\leq n-1} b_iz^i\in \mathbb{F}_{2^m}[z]/(z^n-1)$, the reciprocal polynomial of $g(z)$ is defined as
$$g^*(z)=\sum\limits_{0\leq i\leq n-1} b_iz^{n-i}\pmod{z^n-1}.$$
\end{definition}

This paper is organized as follows. In Section 2, some facts about
normal bases are introduced, and  a few Lemmas which will be used
are listed. In Section 3, we characterize the necessary and
sufficient conditions for $h(z)\equiv g(z) g^*(z)\pmod{(2,z^n-1)}$ for $n=2^s$, and the necessary and sufficient conditions
for a vector to be a corresponding vector of a normal element of
$\mathbb{F}_{2^{2^s}}$ over $\mathbb{F}_2$. In Section 4, we characterize the necessary
and sufficient condition for $h(z)\equiv g(z) g^*(z)\pmod{(2,z^n-1)}$ in the odd case $n$, and present a necessary and
sufficient condition for a vector to be a corresponding vector of a
normal element of $\mathbb{F}_{2^{n}}$ over $\mathbb{F}_2$ where $n$
is odd. Furthermore we give an algorithm to find such normal elements
with a prescribed corresponding vector when $n=2^s$ or $n$ is odd. Some necessary conditions
for a vector to be the corresponding vector of a normal element of
$\mathbb{F}_{2^{n}}$ over $\mathbb{F}_2$ where $4|n$ are obtained in
Section 5. We give the combination method to construct a normal
element with good trace self-orthogonal relation. Especially we can
find a normal element of $\mathbb{F}_{2^{n}}$ over $\mathbb{F}_2$
such that the Hamming weight of its corresponding vector is $3$ for
every $n$ with $4|n$.

\section{Preliminaries}
We make the convention that all polynomials and all arithmetic
about polynomials are assumed to be in the ring
$\mathbb{F}_q[x]/(x^n-1)$ in this paper, where $q$ is a 2-power.
For simplicity, sometimes we use $(f(x),g(x))$ to represent the
greatest common divisor of $f(x)$ and $g(x)$. This section contains
some Lemmas about polynomials in $\mathbb{F}_q[x]/(x^n-1)$.

\begin{theorem}\label{norm1}\cite{gao}
Let $\alpha\in\mathbb{F}_{2^n}$ and
$a_i=\mbox{\textup{Tr}}_{q^n|q}(\alpha\alpha^{q^i})(0\leq i\leq
n-1)$. Then $\alpha$ is a normal element of $\mathbb{F}_{q^n}$ over
$\mathbb{F}_{q}$ if and only if the polynomial $N(x)=\sum\limits_{0\leq i\leq
n-1}a_ix^i\in \mathbb{F}_q[x]$ is relatively prime to $x^n-1$.
\end{theorem}

\begin{theorem}\label{guodu}\cite{perlis}
Let $\alpha$ be a normal element of $\mathbb{F}_{q^n}$ over
$\mathbb{F}_{q}$. $\beta=\sum\limits_{0\leq i\leq
n-1}c_i\alpha^{q^i}$ is also a normal element if and only if the polynomial
$N(x)=\sum\limits_{0\leq i\leq n-1}c_ix^i\in \mathbb{F}_q[x]$ is
relatively prime to $x^n-1$.
\end{theorem}

For symmetric polynomials in $\mathbb{F}_{q}[x]/(x^n-1)$, we have
the following Lemma 2.3.
\begin{lemma}\label{inv}
Suppose $f(x)=\sum\limits_{0\leq i\leq n-1}a_ix^i\in
\mathbb{F}_q[x]$ with $a_0=1$ is symmetric and is relatively prime
to $x^n-1$. Let $f^{-1}(x)=\sum\limits_{0\leq i\leq n-1}b_ix^i\in
\mathbb{F}_q[x]$ be the unique polynomial such that $f(x)\cdot
f^{-1}(x)\equiv 1\pmod{x^n-1}$. Then $f^{-1}(x)$ is
symmetric, relatively prime to $x^n-1$, and its constant term is
$1$.
\end{lemma}
\begin{proof}
Since $f(x)\cdot f^{-1}(x)\equiv 1\pmod{x^n-1}$, there
exists a polynomial $g(x)\in \mathbb{F}_q[x]$ such that
\begin{eqnarray}\label{egcd}
f(x) f^{-1}(x)+g(x) (x^n-1)=1.
\end{eqnarray}

It is easy to see that $(f^{-1}(x),x^n-1)=1$ and the condition that
$f(x)=\sum\limits_{0\leq i\leq n-1}a_ix^i$ is symmetric is equivalent to $f(x)=f^*(x)$.

From equation (\ref{egcd}) we have
\begin{eqnarray}\label{egcd1}
f^*(x) {f^{-1}}^*(x)\equiv 1\pmod{x^n-1}.
\end{eqnarray}
Thus by equations (\ref{egcd}) and (\ref{egcd1}) we get
\begin{eqnarray}\label{egcd2}
f(x)(f^{-1}(x)-{f^{-1}}^*(x))\equiv 0\pmod{x^n-1}.
\end{eqnarray}

Since $(f(x),x^n-1)=1$, from (\ref{egcd2}) we deduce that $f^{-1}(x)={f^{-1}}^*(x)$, which proves that
$f^{-1}(x)$ is symmetric.

Considering the constant terms of equation (\ref{egcd}) , we obtain
\begin{eqnarray}\label{egcd3} a_0b_0+\sum\limits_{0< i,j\leq n-1, i+j=n}a_ib_j=1. \end{eqnarray}

Since $f(x),f^{-1}(x)$ are symmetric, $a_ib_j=a_{n-i}b_{n-j}$ for all $0< i,j\leq n-1$.
Thus for odd $n$, $\sum\limits_{0< i,j\leq n-1, i+j=n}a_ib_j=\sum\limits_{0< i<n/2,0<j\leq n-1, i+j=n}a_ib_j+\sum\limits_{0< i<n/2,0<j\leq n-1, i+j=n}a_{n-i}b_{n-j}=0$. So from (\ref{egcd3}) $a_0b_0=1$, which proves that the constant term $b_0$ of $f^{-1}(x)$ is 1.

For even $n$, we show that $a_{n/2}=b_{n/2}=0$ firstly. Otherwise suppose $a_{n/2}=1$, then the symmetric polynomial
$f(x)=1+x^{n/2}+\sum\limits_{0<i<n/2}a_i(x^i+x^{n-i})$ have a divisor $x-1$, contradictory to the assumption that $(f(x),x^n-1)=1$.
Similarly $b_{n/2}=1$ is impossible.  Therefore $\sum\limits_{0< i,j\leq n-1, i+j=n}a_ib_j=\sum\limits_{0< i<n/2,0<j\leq n-1, i+j=n}a_ib_j+\sum\limits_{0< i<n/2,0<j\leq n-1, i+j=n}a_{n-i}b_{n-j}+a_{n/2}b_{n/2}=0$. So from equation (\ref{egcd3}) we have $a_0b_0=1$ and it follows again that $b_0=1$, which completes the proof.
\end{proof}

Also we give the following obvious lemma without providing the
proof.
\begin{lemma}\label{mult}
Let $f_a(x)=\sum\limits_{0\leq i\leq
n-1}a_ix^i\in\mathbb{F}_q[x]/(x^n-1)$ and $f_b(x)=\sum\limits_{0\leq
i\leq n-1}b_ix^i\in\mathbb{F}_q[x]/(x^n-1)$ be symmetric
polynomials. Then $f_c(x)=f_a(x) f_b(x)=\sum\limits_{0\leq
i\leq n-1}c_ix^i\in\mathbb{F}_q[x]/(x^n-1)$ is also symmetric.
\end{lemma}

\begin{theorem}\label{polybasis}
Let $\beta\in \mathbb{F}_{q^n}$ and $\alpha=\sum\limits_{0\leq i\leq
n-1}c_i\beta^{q^i}, c_i\in \mathbb{F}_q$ with corresponding vectors $a=(a_0,a_1,\cdots,a_{n-1})$ and
$b=(b_0,b_1,\cdots,b_{n-1})$ respectively. Then
$f_a(x)=\sum\limits_{0\leq i\leq n-1}a_ix^i$,
$f_b(x)=\sum\limits_{0\leq i\leq n-1}b_ix^i$, and
$f_c(x)=\sum\limits_{0\leq i\leq n-1}c_ix^i$ satisfy
$$f_a(x)\equiv f_b(x) f_c(x) f_c^*(x)(\mbox{\textup{mod}}\ x^n-1).$$
\end{theorem}
\begin{proof}
Let the circulant matrix over $\mathbb{F}_q$
\begin{eqnarray*}
C=[c_0, c_1, c_2, \dots , c_{n-1}]=\left(\begin{array}{ccccc}
  c_0 & c_1 & c_2 & \dots & c_{n-1} \\
  c_{n-1} & c_0 & c_1 & \dots & c_{n-2} \\
  c_{n-2} & c_{n-1} & c_0 & \dots & c_{n-3} \\
  \vdots & \vdots & \vdots & \vdots & \vdots \\
  c_1 & c_2 & c_3 & \dots & c_0 \\
\end{array}
\right),
\end{eqnarray*}
and denote by $C^T$ the transpose matrix of $C$. Since $\alpha=\sum\limits_{0\leq i\leq n-1}c_i\beta^{q^i},c_i\in
\mathbb{F}_q$, we have
\begin{eqnarray}\label{shu}
\left(
  \begin{array}{c}
    \alpha \\
    \alpha^{q} \\
    \vdots \\
    \alpha^{q^{n-1}} \\
  \end{array}
\right)= C \left(
\begin{array}{c}
    \beta \\
    \beta^{q} \\
    \vdots \\
    \beta^{q^{n-1}} \\
  \end{array}
\right).
\end{eqnarray}


Multiplying (\ref{shu}) by its transpose
$\left( \alpha, \alpha^{q}, \dots,\alpha^{q^{n-1}}
\right)= \left( \beta, \beta^{q},
\dots,\beta^{q^{n-1}}\right)  C^T,
$
we obtain
\begin{equation}\label{jzcf}
\begin{array}{cc}
&\left(\begin{array}{ccccc}
  \alpha^{1+q^0} & \alpha^{1+q^1} & \alpha^{1+q^2} & \dots & \alpha^{1+q^{n-1}} \\
  \alpha^{q+q^0} & \alpha^{q+q^1} & \alpha^{q+q^2} & \dots & \alpha^{q+q^{n-1}} \\
  \vdots & \vdots & \vdots & \vdots & \vdots \\
  \alpha^{q^{n-1}+q^0} & \alpha^{q^{n-1}+q^1} & \alpha^{q^{n-1}+q^2} & \dots & \alpha^{q^{n-1}+q^{n-1}} \\
\end{array}
\right) \\ &= C \left(\begin{array}{ccccc}
  \beta^{1+q^0} & \beta^{1+q^1} & \beta^{1+q^2} & \dots & \beta^{1+q^{n-1}} \\
  \beta^{q+q^0} & \beta^{q+q^1} & \beta^{q+q^2} & \dots & \beta^{q+q^{n-1}} \\
  \vdots & \vdots & \vdots & \vdots & \vdots \\
  \beta^{q^{n-1}+q^0} & \beta^{q^{n-1}+q^1} & \beta^{q^{n-1}+q^2} & \dots & \beta^{q^{n-1}+q^{n-1}} \\
\end{array}
\right)  C^T.
\end{array}
\end{equation}
Taking the trace function $\mbox{Tr}_{q^n|q}$ on every elements of matrices in both sides of
(\ref{jzcf}), we have
\begin{equation}\label{jzcf1}
[a_0 ,a_1 ,a_2 ,\dots , a_{n-1}]
= C [b_0, b_1, b_2, \dots, b_{n-1}] C^T.
\end{equation}

Denote $J_n$ to be the following circulant matrix over $\mathbb{F}_q$,
\begin{eqnarray*}
J_n=[0 ,1 ,0 ,\dots , 0]=\left(\begin{array}{ccccc}
  0 & 1 & 0 & \dots & 0 \\
  0& 0 & 1 & \dots & 0 \\
  \vdots & \vdots & \vdots & \vdots & \vdots \\
  0& 0 & 0 & \dots & 1  \\
  1& 0 & 0 & \dots & 0  \\
\end{array}
\right).
\end{eqnarray*}

Then for a circulant matrix $A$ over
$\mathbb{F}_q$, we have
\begin{eqnarray*}
A=[a_0,\cdots,a_{n-1}] = a_0J_n^0+a_1J_n+a_2J_n^2+\cdots+a_{n-1}J_n^{n-1}.
\end{eqnarray*}

Let $f_a(x)=\sum\limits_{0\leq i\leq n-1}a_ix^i$,
$f_b(x)=\sum\limits_{0\leq i\leq n-1}b_ix^i$ and
$f_c(x)=\sum\limits_{0\leq i\leq n-1}c_ix^i$. Then (\ref{jzcf1})
can be rewritten as
$$f_a(J_n)= f_b(J_n) f_c(J_n) f_c^*(J_n).$$

Since the minimal polynomial of $J_n$ is $x^n-1=0$, by the above equation we can
deduce that
$$f_a(x)\equiv f_b(x) f_c(x) f_c^*(x)(\mbox{mod}\ x^n-1).$$
\end{proof}

\section{The Case $n=2^s\geq 4$}

Suppose $q=2$ in this section. For the following two sets we let
$n=2^s\geq 16$,
\begin{equation*}
G''=\left\{(z_3,\cdots,z_{n/2-1}) \in \mathbb{F}_2^{n/2-3}
\right\},
\end{equation*}

and
\begin{equation*}
H''=\left\{(y_3,\cdots,y_{n/2-1}) \in \mathbb{F}_2^{n/2-3}
\right\}.
\end{equation*}

\begin{lemma}
For $(z_3,\cdots,z_{n/2-1})\in G''$ and $n=2^s$ for $s\geq 4$, let
\begin{equation*}
\left\{
  \begin{array}{ll}
    y_j=h_j(z_3,\cdots,z_{n/2-1})=z_j+z_{j-1}, & 4\leq j\leq n/2-2\ and\ j\ even;\\
    y_3=h_3(z_3,\cdots,z_{n/2-1})=z_3+z_{(n-4)/2}, & j'=3; \\
    y_5=h_5(z_3,\cdots,z_{n/2-1})=z_5+z_4+z_{(n-6)/2}, & j'=5;\\
    y_{j'}=h_{j'}(z_3,\cdots,z_{n/2-1})\\
\ \ \ \ =z_{j'}+z_{j'-1}+z_{(n-1-j')/2}+z_{(j'-1)/2}, & 7\leq j\leq n/2-1\ and\ j'\ odd.
  \end{array}
\right.
\end{equation*}
Then $(h_3,h_4,\cdots,h_{n/2-1})$ defines a one-to-one map from $G''$ to $H''$ when $n=2^s$ for $s\geq 4$.
\end{lemma}
\begin{proof}
For $(z_3,\cdots,z_{n/2-1})\in G''$ and $n=2^s$ for $s\geq 4$, let
\begin{equation}\label{zu0}
\left\{
\begin{array}{ll}
h_3(z_3,\cdots,z_{n/2-1})&=0,\\
\ \ \ \ \vdots\\
h_{n/2-1}(z_3,\cdots,z_{n/2-1})&=0.
\end{array}
\right.
\end{equation}
It suffices to prove the above boolean system (\ref{zu0}) has only
zero solution $(0,0,\cdots,0)\in \mathbb{F}_2^{n/2-3}$.

Firstly, for odd $j=4k+1$ where $9\leq j\leq n/2-1$, and
$j-1,(j-1)/2,(n-j+1)/2$ are even. Thus by (\ref{zu0}),
\begin{equation}
\begin{array}{ll}
h_{j-1}&=z_{j-2}+z_{j-1}=0,\\
h_{(j-1)/2}&=z_{(j-1)/2}+z_{(j-3)/2}=0,\\
h_{(n-j+1)/2}&=z_{(n-j+1)/2}+z_{(n-j-1)/2}=0.
\end{array}
\end{equation}
Hence
\begin{equation}
\begin{array}{ll}
&h_{j-2}+h_{j}\\
&=(z_{j-2}+z_{j-1}+z_{j-3}+z_j)+(z_{(j-1)/2}+z_{(n-j-1)/2}+z_{(j-3)/2}+z_{(n-j+1)/2})\\
&=(h_{j-1}+z_{j-3}+z_j)+(h_{(j-1)/2}+h_{(n-j+1)/2})\\
&=z_{j-3}+z_j.
\end{array}
\end{equation}

So $z_{j-3}=z_j$ for $j=4k+1\geq 9$.

By $h_3=h_5=0$ , $h_4=z_3+z_4=0$ and
$h_{(n-4)/2}=z_{(n-4)/2}+z_{(n-6)/2}=0$, we have
$h_3+h_5=z_3+z_{(n-4)/2}+z_5+z_4+z_{(n-6)/2}=h_4+h_{(n-4)/2}+z_5=0$.
So $z_5=0$.

By $h_{n/2-1}=z_{n/4}+z_{n/4-1}+z_{n/2-1}+z_{n/2-2}=0$ and
$h_{n/4}=z_{n/4}+z_{n/4-1}=0$. we have $z_{n/2-1}=z_{n/2-2}$.

So
\begin{equation*}
\left\{
\begin{array}{ll}
z_5&=0,\\
h_{2j}&=z_{2j-1}+z_{2j}=0,\ \ \ for \ all\ 4\leq 2j\leq n/2-1 \\
z_{j-3}&=z_j.\ \ \ for \ all\ 9\leq j=4k+1\leq n/2-1
\end{array}
\right.
\end{equation*}

Thus $z_6=0$ since $z_5=0$ and $h_6=z_5+z_6=0$. By $z_6=0$ we have
$z_9=z_6=0$ since $z_{j-3}=z_j \ for \ all\ 9\leq j=4k+1\leq
n/2-1$. Thus $z_{10}=0$ since $z_9=0$ and $h_{10}=z_{10}+z_9=0$,
and so on. By this procedure we deduce that
\begin{equation}\label{zu00}
\left\{
\begin{array}{ll}
&z_{i}=z_{i-1}=0,\ \ \ i\equiv 2\pmod{4},\ 3\leq i\leq n/2-1, \\
&z_{n/2-1}=0.
\end{array}
\right.
\end{equation}

Now let $j=4k+1$ and $k\equiv 1\pmod{2}$. Then $j-1\equiv
4\pmod{8}$, and $j\equiv 1\pmod{4}$, $(j-1)/2\equiv
2\pmod{4}$, $(n-j-1)/2\equiv 1\pmod{4}$. Thus for
$j>5$, $z_j=z_{(j-1)/2}=z_{(n-j-1)/2}=0$ by (\ref{zu00}), and
$h_{j}=z_j+z_{j-1}+z_{(j-1)/2}+z_{(n-j-1)/2}=0$, which gives
$z_{j-1}=z_{4k}=0$ for $k>1$. In the case of $j=5$,
$z_5=z_{(n-5-1)/2}=0$ by (\ref{zu00}), and
$h_5=z_5+z_4+z_{(n-5-1)/2}=0$, so similarly $z_4=0$. Generally by
$h_{4k}=z_{4k}+z_{4k-1}=0$, we have
\begin{equation}\label{zu01}
z_{i}=z_{i-1}=0,\ \ \ i\equiv 4\pmod{8},\ 3\leq i\leq n/2-1.
\end{equation}

Assume $j=4k+1$ and $k\equiv 2\pmod{4}$. Similarly by
considering
$h_{j}=z_j+z_{j-1}+z_{(j-1)/2}+z_{(n-j-1)/2}=z_{j-1}=0$, we can
get
\begin{equation}\label{zu02}
z_{i}=z_{i-1}=0,\ \ \ \ i\equiv 8\pmod{16},\ 3\leq i\leq n/2-1.
\end{equation}

Generally when $j=4k+1$ and $k\equiv 2^{t}\pmod{2^{t+1}}$,
where $2^t\leq n/16$, i.e. $0\leq t\leq s-4$, we have
\begin{equation}\label{zu03}
z_{i}=z_{i-1}=0,\ \ \ i\equiv 2^{t+2}\pmod{2^{t+3}},\ 3\leq i\leq n/2-1.
\end{equation}

By (\ref{zu00}),(\ref{zu01}),(\ref{zu02}),(\ref{zu03}) and
$$\{z_3,z_4,\cdots,z_{n/2-2}\}=\bigcup\limits_{t=1}\limits^{s-2} \bigcup \limits_{3\leq i\equiv 2^t (\mbox{mod}\ 2^{t+1})\leq n/2-1, }{\{z_i,z_{i-1}\}},$$
we know that the only solution of the boolean system (\ref{zu0})
is
$$\{z_3,z_4,\cdots,z_{n/2-2},z_{n/2-1}\}=(0,0,\cdots,0)\in \mathbb{F}_2^{n/2-3}.$$

\end{proof}

\begin{theorem}\label{main2}
Let $h(z)=\sum\limits_{0\leq i\leq n-1}a_iz^i \in
\mathbb{F}_2[z]/(z^n-1), n=2^s$ for $s\geq 2$, and $a_0=1$. Then
$h(z)=g(z)\cdot g(z)^*$ for some $g(z)=\sum\limits_{0\leq i\leq
n-1}b_iz^i \in \mathbb{F}_2[z]/(z^n-1)$ if and only if $h(z)\in
H$, where
\begin{equation}
H=\left\{h(z)=\sum\limits_{0\leq i\leq n-1}a_iz^i \in \mathbb{F}_2[z]/(z^n-1)
\left|
\begin{array}{ll}
&\sum\limits_{0\leq i\leq n/2-1, (i,2)=1}a_i=0, \\
&\ a_0=1,a_{n/2}=0,\\
&\ a_i=a_{n-i}\ for\ all\ 1\leq i\leq n/2-1.
\end{array}
\right.
\right\}
\end{equation}

Furthermore, every $h(z)\in H$ has a unique factorization
$h(z)=g(z)\cdot g^*(z)$ for some $g(z)\in G$,  where
\begin{equation}
G=\left\{g(z)=\sum\limits_{0\leq i\leq n-1}b_iz^i \in \mathbb{F}_2[z]/(z^n-1)
\left|
\begin{array}{ll}
&b_0=1,b_{n-1}=0,\\
&b_2=b_{n-3}=0,\\
&b_i=b_{n-1-i}\ for\ all\ i=1, 3, 4, \cdots, n/2-1.
\end{array}
\right.
\right\}
\end{equation}
\end{theorem}

\begin{proof}
Firstly it is easy to deduce that the condition that $h(z)=g(z) g^*(z)$ for some $g(z)\in \mathbb{F}_2[z]/(z^n-1)$ is equivalent to the following equation system having a solution:
\begin{equation} \label{zu1}
\left\{
\begin{array}{ll}
h_0(z_0,\cdots,z_{n-1})&=z_0+ \cdots + z_{n-1} =a_0=1,\\
h_1(z_0,\cdots,z_{n-1})&=z_0z_1+z_1z_2+z_2z_3+ \cdots + z_{n-1}z_0 =a_{n-1},\\
&\vdots\\
h_{n/2}(z_0,\cdots,z_{n-1})&=z_0z_{n/2}+z_1z_{n/2+1}+ \cdots + z_{n/2-1}z_{n-1}\\
&+z_0z_{n/2}+z_1z_{n/2+1}+ \cdots + z_{n/2-1}z_{n-1} =a_{n/2},\\
&\vdots\\
h_{n-1}(z_0,\cdots,z_{n-1})&=z_0z_{n-1}+z_1z_0+z_2z_1+ \cdots + z_{n-1}z_{n-2} =a_{1}.
\end{array}
\right.
\end{equation}

In the above equation system, it is obvious that for $1\leq j\leq n/2-1$,
\begin{equation*}
\begin{array}{ll}
&h_j(z_0,\cdots,z_{n-1})=a_{n-j}\\
&=z_0z_j+z_1z_{1+j}+z_2z_{2+j}+ \cdots + z_{n-1}z_{j-1} \\
&=z_0z_{n-j}+z_1z_{n-j+1}+\cdots +z_{n-1}z_{n-j-1}\\
&=a_j=h_{n-j}(z_0,\cdots,z_{n-1}).
\end{array}
\end{equation*}

And for $i=n/2$,
\begin{equation*}
\begin{array}{ll}
&h_{n/2}(z_0,\cdots,z_{n-1})=a_{n/2}\\
&=z_0z_{n/2}+z_1z_{n/2+1}+ \cdots + z_{n/2-1}z_{n-1}+z_0z_{n/2}+z_1z_{n/2+1}+ \cdots + z_{n/2-1}z_{n-1} \\
&=0.
\end{array}
\end{equation*}

Thus a necessary condition for $h(z)=g(z) g^*(z)$ is $a_t=a_{n-t}$ for all $1\leq t\leq n/2-1$, and $a_{n/2}=0$. Thus (\ref{zu1}) is equivalent to
\begin{equation} \label{zu2}
\left\{
\begin{array}{ll}
h_0(z_0,\cdots,z_{n-1})&=z_0+ \cdots + z_{n-1} =1,\\
h_1(z_0,\cdots,z_{n-1})&=z_0z_1+z_1z_2+z_2z_3+ \cdots + z_{n-1}z_0 =a_1,\\
&\vdots\\
h_{n/2-1}(z_0,\cdots,z_{n/2-1})&=z_0z_{n/2-1}+z_1z_{n/2}+ \cdots + z_{n-1}z_{n/2-2} =a_{n/2-1}.
\end{array}
\right.
\end{equation}

(Necessity) It is left to prove the necessary condition that  $\sum\limits_{0\leq i\leq n/2-1, (i,2)=1}a_i=0$.

If $h(z)=g(z) g^*(z)$ for some
$g(z)=\sum\limits_{0\leq i\leq n-1}b_iz^i \in \mathbb{F}_2[z]/(z^n-1)$, then  $(b_0,b_1,\cdots,b_{n-1})\in \mathbb{F}_2^n$ is a solution of
the equation system (\ref{zu2}).

 Since $h_0(b_0,\cdots,b_{n-1})=b_0+ \cdots + b_{n-1} =\sum\limits_{(i,2)>1}b_i+\sum\limits_{(i,2)=1}b_i=1$ and $4|2^s=n$, $$(\sum\limits_{(i,2)>1}b_i)\cdot (\sum\limits_{(i,2)=1}b_i)=0.$$ So we have
\begin{equation*}
\begin{array}{ll}
&(\sum\limits_{(i,2)>1}b_i)\cdot (\sum\limits_{(i,2)=1}b_i)\\
&=\sum\limits_{(j,2)=1,1\leq j\leq n/2-1}\sum\limits_{0\leq i\leq n-1}b_ib_{i+j}\\
&=b_0b_1+b_1b_2+b_2b_3+ \cdots + b_{n-1}b_0\\
&\ \ +b_0b_3+b_1b_4+b_2b_5+ \cdots + b_{n-1}b_2\\
&\ \ \ \vdots\\
&\ \ +b_0b_{n/2-1}+b_1b_{n/2}+b_2b_4+ \cdots + b_{n-1}b_{n/2-2}\\
&=h_1(b_0,\cdots,b_{n-1})+h_3(b_0,\cdots,b_{n-1})+\cdots+h_{n/2-1}(b_0,\cdots,b_{n-1})\\
&=a_1+a_3+a_5+\cdots+a_{n/2-1}\\
&=\sum\limits_{0\leq i\leq n/2-1, (i,2)=1}a_i=0.
\end{array}
\end{equation*}

(Sufficiency) It is easy to verify the sufficiency for two cases
$s=2,3$, i.e. $n=4,8$. So we assume $s\geq 4$ in the following. It
is obvious that the sets $H$ and $G$ all have $2^{n/2-2}$
elements. Thus it is sufficient to prove the following map
$$P:\ P(g(z))=g(z)\cdot g(z)^*\ for\ g(z)\in G$$
is an injective map from $G$ to $H$.

By (\ref{zu1}), $g\cdot g^*$ is equivalent to a boolean equation system. So we should prove the map
$P':P'(z_0,z_1,\cdots,z_{n-1})=(h_0(z_0,\cdots,z_{n-1}),\cdots,h_{n-1}(z_0,\cdots,z_{n-1}))$ is an injective map from $G'$ to $H'$, where
$$h_j(z_0,\cdots,z_{n-1})=\sum\limits_{0\leq i\leq n-1}z_iz_{i+j}, \ 0\leq j\leq n-1.$$

\noindent And

\begin{equation*}
G'=\left\{(z_0,\cdots,z_{n-1}) \in \mathbb{F}_2^n
\left|
\begin{array}{ll}
&z_0=1,z_{n-1}=0,\\
&z_2=z_{n-3}=0,\\
&z_i=z_{n-1-i}\ for\ all\ 1\leq i\neq 2\leq n/2-1.
\end{array}
\right.
\right\}
\end{equation*}

and

\begin{equation*}
H'=\left\{(y_0,\cdots,y_{n-1}) \in \mathbb{F}_2^n
\left|
\begin{array}{ll}
&\sum\limits_{0\leq i\leq n/2-1, (i,2)=1}y_i=0, \\
&\ y_0=1,y_{n/2}=0,\\
&\ y_i=y_{n-i}\ for\ all\ 1\leq i\leq n/2-1.
\end{array}
\right.
\right\}
\end{equation*}

Now we will prove the quadratic boolean system (\ref{zu1}) is in fact linear for $(z_0,\cdots,z_{n-1})\in G'$. Firstly we discuss $j\in [1,n/2-1]$ in two cases.

1) For every $1\leq j\leq n/2-1$ and $(j,2)>1$, the $n$ items $z_iz_{i+j}, 0\leq i\leq n-1$ in $h_j(z_0,\cdots,z_{n-1})=\sum\limits_{0\leq i\leq n-1}z_iz_{i+j}$ can be separated into $n/2$ pairs. Every pairs are
$$z_iz_{i+j},\ \ \mbox{and}\ \ z_{n-i-1-j}z_{n-i-1},\ 0\leq i\leq (n-j-1)/2, \mbox{or}\ i\geq n-(j+1)/2.$$

Since $(j,2)>1$, the two items $z_iz_{i+j}$ and  $z_{n-1-i-j}z_{n-1-i}$ are different for all $0\leq i\leq (n-j-1)/2$, and $i\geq n-(j+1)/2$. Otherwise $i\equiv n-i-1-j\pmod{n}$. So
the congruence equation $2i\equiv n-1-j\pmod{n}$ has  solutions, which is impossible for $(j,2)>1$.

Because $(z_0,\cdots,z_{n-1})\in G'$, we have $z_i=z_{n-i-1}$ and $z_{i+j}=z_{n-1-i-j}$ for $i\neq 0$. When $i=0$,$z_0=1\neq z_{n-1}=0$, the item pairs $z_0z_j=z_j$ and $z_{n-1-j}z_{n-1}=0$. When $i=n-j$, the item pairs $z_iz_{i+j}=z_{n-j}z_0=z_{n-j}=z_{j-1}$ and $z_{j-1}z_{n-1}=0$. So in the even $j$ case,
\begin{equation*}
\begin{array}{ll}
h_j(z_0,\cdots,z_{n-1})&=\sum\limits_{0\leq i\leq n-1}z_iz_{i+j}\\
&=\sum\limits_{0\leq i\leq (n-j-1)/2,i\geq n-(j+1)/2,i\neq 0,n-j}(z_iz_{i+j}+z_{n-i-1-j}z_{n-i-1})\\
&\ \ + (z_0z_j+z_{n-1-j}z_{n-1}) +(z_{n-j}z_0+z_{n-1-(n-j)}z_{n-1})\\
&=(z_0z_j+z_{n-1-j}z_{n-1}) +(z_{n-j}z_0+z_{j-1}z_{n-1})\\
&=z_j+z_{j-1}.
\end{array}
\end{equation*}

2) For every $1\leq j\leq n/2-1$ and $(j,2)=1$, there are two solutions $i=(n-1-j)/2,n-(j+1)/2$ for the congruence equation $i\equiv n-i-1-j\pmod{n}$. Thus the item $z_iz_{i+j}$ is equal to $z_{n-i-1-j}z_{n-i-1}$ for $i=(n-1-j)/2,n-(j+1)/2$. In these two cases, $z_iz_{i+j}=z_{(n-1-j)/2}$ when $i=(n-1-j)/2$, and $z_iz_{i+j}=z_{n-(1+j)/2}=z_{(1+j)/2-1}=z_{(j-1)/2}$ when $i=n-(1+j)/2$.

And for $0\leq i< (n-j-1)/2,i> n-(j+1)/2, i\neq 0,n-j$, the item pair $z_iz_{i+j}$ and  $z_{n-i-1-j}z_{n-i-1}$ are different. In these cases $z_iz_{i+j}+z_{n-i-1-j}z_{n-i-1}=0$.

Similarly, when $i=0$, the item pairs $z_0z_j=z_j$ and $z_{n-1-j}z_{n-1}=0$. When $i=n-j$, the item pairs $z_iz_{i+j}=z_{n-j}z_0=z_{n-j}=z_{j-1}$ and $z_{j-1}z_{n-1}=0$.

So the $n$ items $z_iz_{i+j}, 0\leq i\leq n-1$ in $h_j(z_0,\cdots,z_{n-1})=\sum\limits_{0\leq i\leq n-1}z_iz_{i+j}$ can be separated into $(n-2)/2+1$ pairs. Therefore for odd $j$,

\begin{equation*}
\begin{array}{ll}
h_j(z_0,\cdots,z_{n-1})
&=\sum\limits_{0\leq i\leq n-1}z_iz_{i+j}\\
&=\sum\limits_{0\leq i< (n-j-1)/2, i> n-(j+1)/2, i\neq 0,n-j }(z_iz_{i+j}+z_{n-i-1-j}z_{n-i-1})\\
&\ \ + z_{(n-1-j)/2}z_{(n+j-1)/2}+ z_{n-(1+j)/2}z_{(j-1)/2}\\
&\ \ + (z_0z_j+z_{n-1-j}z_{n-1}) +(z_{n-j}z_0+z_{n-1-(n-j)}z_{n-1})\\
&=\left\{
\begin{array}{ll}
z_{(n-1-j)/2}z_{(n-1-j)/2}+z_{(j-1)/2}z_{(j-1)/2}\ \ &j\neq 1\\
\ \ +(z_0z_j+z_{n-1-j}z_{n-1}) +(z_{n-j}z_0+z_{j-1}z_{n-1}) \\
z_{(n-1-j)/2}z_{(n-1-j)/2}+z_{n-1}z_0\ \ &j= 1\\
\ \ +(z_0z_j+z_{n-1-j}z_{n-1}) +(z_{n-1}z_0+z_{j-1}z_{n-1}) \\
\end{array}
\right.\\

&=\left\{
\begin{array}{ll}
z_{(n-1-j)/2}+z_{(j-1)/2}+z_j+z_{j-1}\ \ \ &j\neq 1\\
z_{(n-1-j)/2}+z_j\ \ \  &j=1.
\end{array}
\right.
\end{array}
\end{equation*}

3) For $j=0$, we know that if $(z_0,\cdots,z_{n-1})\in G'$, $z_i+z_{n-1-i}=0$ for $1\leq i\leq n/2-1$ and $z_0+z_{n-1}=1$, so
$h_0(z_0,\cdots,z_{n-1})=z_0+\cdots+z_{n-1}=(z_0+z_{n-1})+\sum\limits_{1\leq i\leq n/2-1}(z_i+z_{n-1-i})=1$.

4) For $n/2\leq j\leq n-1$, it is easy to see that $h_{n/2}(z_0,\cdots,z_{n-1})=0$ and $h_j(z_0,\cdots,z_{n-1})=h_{n-j}(z_0,\cdots,z_{n-1})$.

To sum up, for $(z_0,\cdots,z_{n-1})\in G'$ we have
\begin{equation}\label{zu3}
\left\{
\begin{array}{ll}
y_0&=h_0(z_0,\cdots,z_{n-1})=1,\\
y_1&=h_1(z_0,\cdots,z_{n-1})=z_{(n-2)/2}+z_1,\\
y_2&=h_2(z_0,\cdots,z_{n-1})=z_2+z_1=z_1,\\
y_3&=h_3(z_0,\cdots,z_{n-1})=z_3+z_2+z_{(n-4)/2}+z_1=z_1+z_3+z_{(n-4)/2},\\
y_4&=h_4(z_0,\cdots,z_{n-1})=z_4+z_3,\\
y_5&=h_5(z_0,\cdots,z_{n-1})=z_5+z_4+z_{(n-6)/2}+z_2=z_5+z_4+z_{(n-6)/2},\\
y_6&=h_6(z_0,\cdots,z_{n-1})=z_6+z_5,\\
y_7&=h_7(z_0,\cdots,z_{n-1})=z_7+z_6+z_{(n-8)/2}+z_3,\\
&\vdots\\
y_j&=h_j(z_0,\cdots,z_{n-1})=z_j+z_{j-1}, j>2\ even,\\
&\vdots\\
y_j'&=h_j'(z_0,\cdots,z_{n-1})=z_{j'}+z_{j'-1}+z_{(n-1-j')/2}+z_{(j'-1)/2}, j'>5\ odd,\\
&\vdots\\
y_{n/2-1}&=h_{n/2-1}(z_0,\cdots,z_{n-1})=z_{n/4}+z_{n/4-1}+z_{n/2-1}+z_{n/2-2},\\
y_{n/2}&=h_{n/2-1}(z_0,\cdots,z_{n-1})=0,\\
y_{n/2+1}&=h_{n/2+1}(z_0,\cdots,z_{n-1})=y_{n/2-1},\\
&\vdots\\
y_{n/2+j}&=h_{n/2+j}(z_0,\cdots,z_{n-1})=y_{n/2-j},\\
&\vdots\\
y_{n-1}&=h_{n-1}(z_0,\cdots,z_{n-1})=y_{1}.\\
\end{array}
\right.
\end{equation}

By (\ref{zu3}), we get
\begin{equation}\label{y1}
\begin{array}{ll}
&\sum\limits_{1\leq j\leq n/2-1,(j,2)=1}y_j\\
&=h_1+h_3+\cdots+h_{n/2-1}\\
&=\sum\limits_{1\leq i\leq n/2-1,i\neq 2}z_i+\sum\limits_{1\leq i\leq n/2-1,i\neq 2}z_i\\
&=0.
\end{array}
\end{equation}

So we have proved that for all $(z_0,\cdots,z_{n-1})\in G': \ P'(z_0,\cdots,z_{n-1})\in H'$, this means that for all $g(z)\in G: \ P(g(z))\in H$.

On the other hand, by (\ref{zu3}), (\ref{y1}),  and Lemma 3.1,
$P(g(z))s$ are different for different $g(z)s\in
G$.

So we have proved that $P(G)=H$, and every $h$ have a unique factorization $h=g\cdot g^*$ for some $g\in G$.
\end{proof}

\begin{remark}
It should be noted that the factorization is not unique if $g(z)\notin G$. Also if $n\neq 2^s$, then there is possible no factorization for $h(z)\in H$.
\end{remark}

\begin{lemma}\label{mult1}
Let $n$ be even, $f_a(x)=\sum\limits_{0\leq i\leq
n-1}a_ix^i\in\mathbb{F}_2[x]/(x^n-1)$ and
$f_b(x)=\sum\limits_{0\leq i\leq
n-1}b_ix^i\in\mathbb{F}_2[x]/(x^n-1)$ are symmetric polynomials
with $a_0=b_0=1$, $a_{n/2}=b_{n/2}=0$. Suppose $f_c(x)=f_a(x)
f_b(x)=\sum\limits_{0\leq i\leq
n-1}c_ix^i\in\mathbb{F}_2[x]/(x^n-1)$. Then $c_0=1,c_{n/2}=0$, and
$$\sum\limits_{1\leq i\leq n/2-1, (i,2)=1}c_i=\sum\limits_{1\leq
i\leq n/2-1, (i,2)=1}(a_i+b_i).$$
\end{lemma}
\begin{proof}
Since $f_a(x)$ and $f_b(x)$ are symmetric, $f_c(x)$ is symmetric
by Lemma \ref{mult}. Also
\begin{eqnarray*}
\sum\limits_{1\leq i\leq n/2-1, (i,2)=1}a_i &=& \sum\limits_{n/2+1\leq i\leq n-1, (i,2)=1}a_i,
\sum\limits_{1\leq i\leq n/2-1, (i,2)=2}a_i = \sum\limits_{n/2+1\leq i\leq n-1, (i,2)=2}a_i, \\
\sum\limits_{1\leq i\leq n/2-1, (i,2)=1}b_i &=& \sum\limits_{n/2+1\leq i\leq n-1, (i,2)=1}b_i,
\sum\limits_{1\leq i\leq n/2-1, (i,2)=2}b_i =
\sum\limits_{n/2+1\leq i\leq n-1, (i,2)=2}b_i.
\end{eqnarray*}

Note that when $n$ is even, the odd exponent term $c_ix^i$ in
$f_c(x)$ can only be obtained by multiplying $a_jx^j$ in $f_a(x)$
and $b_kx^k$ in $f_b(x)$ under the condition that $j\neq
k(\mbox{mod}\ 2)$. Hence
\begin{eqnarray*}
\sum\limits_{1\leq i\leq n/2-1, (i,2)=1}c_i
&=&2\cdot\sum\limits_{1\leq j\leq n/2-1, (j,2)=1}a_j\cdot \sum\limits_{1\leq k\leq n/2-1, (k,2)=2}b_i\\
&\ &\quad +2\cdot\sum\limits_{1\leq k\leq n/2-1, (k,2)=1}b_i\cdot \sum\limits_{1\leq j\leq n/2-1, (j,2)=2}a_i\\
&\ &\quad +\sum\limits_{1\leq i\leq n/2-1, (i,2)=1}a_i + \sum\limits_{1\leq i\leq n/2-1, (i,2)=1}b_i\\
&=&\sum\limits_{1\leq i\leq n/2-1, (i,2)=1}(a_i+b_i).
\end{eqnarray*}
Similarly, $c_0=1$ and $c_{n/2}=0$ can be easily verified.
\end{proof}

\begin{theorem}\label{2powercase}
For $n=2^s\geq 4$, there exists a normal element of
$\mathbb{F}_{2^{n}}$ over $\mathbb{F}_2$ corresponding to vector
$(a_0,a_1,\cdots,a_{n-1})\in \mathbb{F}_2^n$, if and only if

1) $a_0=1$, $a_{n/2}=0$,

2) $a_i=a_{n-i}$ for all $1\leq i\leq n/2-1$,

3) $\sum\limits_{1\leq i\leq n/2-1, (i,2)=1}a_i=1$.
\end{theorem}
\begin{proof}
(Necessity) Let $\alpha$ be a normal element of $
\mathbb{F}_{2^{n}}$ over $\mathbb{F}_2$, and its corresponding
vector is $(a_0,a_1,\cdots,a_{n-1})\in \mathbb{F}_2^n$. Then
obviously
$a_0=\mbox{Tr}_{2^n|2}(\alpha\cdot\alpha^{2^0})=\mbox{Tr}_{2^n|2}(\alpha)=1$,
$a_{n/2}=\mbox{Tr}_{2^n|2}(\alpha\cdot\alpha^{2^{n/2}})=\mbox{Tr}_{2^{n/2}|2}(\mbox{Tr}_{2^n|2^{n/2}}(\alpha\cdot\alpha^{2^{n/2}}))
=\mbox{Tr}_{2^{n/2}|2}(0)=0$, and
$a_i=\mbox{Tr}_{2^n|2}(\alpha\cdot\alpha^{2^i})=\mbox{Tr}_{2^n|2}(\alpha\cdot\alpha^{2^{n-i}})=a_{n-i}$
for all $1\leq i\leq n/2-1$.

Let $u=\sum\limits_{0\leq i\leq n-1,
(i,2)=1}\alpha^{2^i}$. Then $u^2=\sum\limits_{0\leq j\leq
n-1,(j,2)=2}\alpha^{2^j}$, and $\mbox{Tr}_{2^n|2}(\alpha)=u+u^2$.
Because $\mbox{Tr}_{2^n|2}(\alpha)=a_0=1$, $u+u^2=1=u(u+1)$, and
$u^2=u+1$. Hence
\begin{eqnarray*}
\sum\limits_{1\leq i\leq n/2-1, (i,2)=1}a_i
&=&\sum\limits_{1\leq i\leq n/2-1, (i,2)=1}\sum\limits_{0\leq j\leq n-1}(\alpha^{1+2^i})^{2^j}\\
&=&(\sum\limits_{0\leq i\leq n-1, (i,2)=1}\alpha^{2^i})\cdot (\sum\limits_{0\leq j\leq n-1,(j,2)=2}\alpha^{2^j})\\
&=&u(u+1)\\
&=&1.
\end{eqnarray*}

(Sufficiency) Suppose $(a_0,a_1,\cdots,a_{n-1})\in \mathbb{F}_2^n$
satisfies conditions 1)-3), we find a normal element $\alpha$ of
$\mathbb{F}_{2^{n}}$ over $\mathbb{F}_2$ such that
$\mbox{Tr}_{2^n|2}(\alpha^{1+2^i})=a_i$ for all $0\leq i\leq n-1$.
Let $f_a(x)=\sum\limits_{0\leq i\leq n-1}a_ix^i$.

The well-known normal basis theorem Theorem \ref{nbth} tells us that
there exists a normal element $\beta$ of $\mathbb{F}_{2^{n}}$ over
$\mathbb{F}_2$. Let its corresponding polynomial be
$f_b(x)=\sum\limits_{0\leq i\leq n-1}b_ix^i$ with
$b_i=\mbox{Tr}_{2^n|2}(\beta^{1+2^i})$. From the above necessity
part proof, we see that $f_b(x)$ is symmetric with $b_0=1$, and
$\sum\limits_{1\leq i\leq n/2-1, (i,2)=1}b_i=1$.

Obviously $(f_b(x),x^n-1)=1$ when $n=2^s$. Let
$f_b^{-1}(x)(\mbox{mod}\ x^n-1)=\sum\limits_{0\leq i\leq
n-1}b_i'x^i$. By Lemma \ref{inv}, $b_0'=1$, $f_b^{-1}(x)$ is
symmetric and relatively prime to $x^n-1$, which also implies that
$b_{n/2}'=0$.

Since $f_b(x)\cdot f_b^{-1}(x)=1$ in the polynomial ring
$\mathbb{F}_2[x]/(x^n-1)$, by Lemma \ref{mult1} we have
$\sum\limits_{1\leq i\leq n/2-1, (i,2)=1}(b_i+b_i')=0$. Therefore
$$\sum\limits_{1\leq i\leq n/2-1, (i,2)=1}b_i'=\sum\limits_{1\leq
i\leq n/2-1, (i,2)=1}b_i=1.$$

Let $h(x)=f_a(x)\cdot f_b^{-1}(x)=\sum\limits_{0\leq i\leq
n-1}h_ix^i\in \mathbb{F}_2[x]/(x^n-1)$. Again by Lemma
\ref{mult1}, $h(x)$ is symmetric, $h_0=1$, $h_{n/2}=0$, and
$$\sum\limits_{1\leq i\leq n/2-1, (i,2)=1}h_i=\sum\limits_{1\leq
i\leq n/2-1, (i,2)=1}(b_i'+a_i)=0.$$

So $h(x)$ is in the set $H$ defined by Theorem \ref{main2}. Thus
we can find a unique solution $g(x)=\sum\limits_{0\leq i\leq
n-1}c_ix^i$ in the set $G$ defined in Theorem \ref{main2} such
that $h(x)=g(x)g^*(x)$. It can be easily verified that
$(g(x),x^n-1)=1$.

Now let $\alpha=\sum\limits_{0\leq i\leq n-1}c_i\beta^{2^i}$.
Theorem \ref{guodu} shows that $\alpha$ is a normal element of
$\mathbb{F}_{2^{n}}$ over $\mathbb{F}_2$. Furthermore
$h(x)=g(x) g^*(x)$ means that $f_a(x)=f_b(x) g(x)
g^*(x)$. By Lemma \ref{polybasis}, we know that the corresponding
vector of the normal element $\alpha$ is the given vector
$(a_0,a_1,\cdots,a_{n-1})$.
\end{proof}

\section{The Case odd $n$}
\begin{theorem}\label{oddfenjie}
Let $h(z)=\sum\limits_{0\leq i\leq n-1}a_iz^i \in
\mathbb{F}_2[z]/(z^n-1)$ for odd $n$. Then $h(z)=g(z) g^*(z)$
for some $g(z)=\sum\limits_{0\leq i\leq n-1}b_iz^i \in
\mathbb{F}_2[z]/(z^n-1)$ if and only if $h(z)$ is symmetric.

Furthermore, every symmetric polynomial $h(z)=\sum\limits_{0\leq
i\leq n-1}a_iz^i\in \mathbb{F}_2[z]/(z^n-1)$ has a factorization
$h(z)=g(z) g(z)^*$ such that
$$g(z)=\sum\limits_{0\leq i\leq n-1}a_{2i(\mbox{\textup{mod}}\ n)}z^i.$$
\end{theorem}
\begin{proof}
(Necessity) It is easy to deduce that the condition that
$h(z)=g(z) g(z)^*$ for some $g(z)\in \mathbb{F}_2[z]/(z^n-1)$
is equivalent to the following equation system having a solution:
\begin{equation} \label{zu11}
\left\{
\begin{array}{ll}
h_0(z_0,\cdots,z_{n-1})&=z_0+ \cdots + z_{n-1} =a_0,\\
h_1(z_0,\cdots,z_{n-1})&=z_0z_1+z_1z_2+z_2z_3+ \cdots + z_{n-1}z_0 =a_{n-1},\\
&\vdots\\
h_{n-1}(z_0,\cdots,z_{n-1})&=z_0z_{n-1}+z_1z_0+z_2z_1+ \cdots +
z_{n-1}z_{n-2} =a_{1}.
\end{array}
\right.
\end{equation}

In the above equation system, it is obvious that for $1\leq j\leq
(n-1)/2$,
\begin{equation*}
\begin{array}{ll}
&h_j(z_0,\cdots,z_{n-1})=a_{n-j}\\
&=z_0z_j+z_1z_{1+j}+z_2z_{2+j}+ \cdots + z_{n-1}z_{t-1} \\
&=z_0z_{n-j}+z_1z_{n-j+1}+\cdots +z_{n-1}z_{n-t-1}\\
&=a_j=h_{n-j}(z_0,\cdots,z_{n-1}).
\end{array}
\end{equation*}

(Sufficiency) Suppose that $h(z)=\sum\limits_{0\leq i\leq
n-1}a_iz^i \in \mathbb{F}_2[z]/(z^n-1)$ is symmetric. Let
$g(z)=\sum\limits_{0\leq i\leq n-1}b_iz^i$ where
$b_i=a_{2i(\mbox{mod} n)}$.

It is easy to see that the set $\{a_0,\cdots,a_{n-1}\}$ is a
reassignment of $\{b_0,\cdots,b_{n-1}\}$ because $(n,2)=1$. Since
$a_i=a_{n-i}$ for $1\leq i\leq (n-1)/2$, we have $b_i=b_{n-i}$ for
$1\leq i\leq (n-1)/2$. Thus $g(z)=g^*(z)$. So
$$g(z)\cdot g^*(z)=(g(z))^2=h(z).$$
\end{proof}

\begin{theorem}\label{oddcase}
For odd $n$, there exists a normal element of $\mathbb{F}_{2^{n}}$
over $\mathbb{F}_2$ corresponding to a vector
$(a_0,a_1,\cdots,a_{n-1})\in \mathbb{F}_2^n$, if and only if
$f_a(x)=\sum\limits_{0\leq i\leq n-1}a_ix^i$ is symmetric and
$(f_a(x),\ x^n-1)=1$.
\end{theorem}
\begin{proof}
(Necessity) Since $a_i=Tr_{2^n|2}(\alpha^{1+2^i})=Tr_{2^n|2}(\alpha^{1+2^{n-i}})$ for all $i\in [1,n-1]$, $(a_0,a_1,$ $\cdots,a_{n-1})$
is symmetric. By Theorem \ref{norm1}, we have $(f_a(x),\ x^n-1)=1$.

(Sufficiency) Suppose that $(a_0,a_1,\cdots,a_{n-1})\in
\mathbb{F}_2^n$ satisfies $(f_a(x),\ x^n-1)=1$, we find a normal
element $\alpha$ of $\mathbb{F}_{2^{n}}$ over $\mathbb{F}_2$ such
that $\mbox{Tr}_{2^n|2}(\alpha^{1+2^i})=a_i$ for all $0\leq i\leq
n-1$.

By the well-known normal basis theorem Theorem \ref{nbth}, there exists
a normal element $\beta$ of $\mathbb{F}_{2^{n}}$ over
$\mathbb{F}_2$. Let its corresponding polynomial be
$f_b(x)=\sum\limits_{0\leq i\leq n-1}b_ix^i$ with
$b_i=\mbox{Tr}_{2^n|2}(\beta^{1+2^i})$. From the above necessity
part proof, we see that $f_b(x)$ is symmetric and $(f_b(x),\
x^n-1)=1$.

Let $f_b^{-1}(x)(\mbox{mod}\ x^n-1)=\sum\limits_{0\leq i\leq
n-1}b_i'x^i$. By Lemma \ref{inv}, $f_b^{-1}(x)$ is symmetric and
relatively prime to $x^n-1$.

So $h(x)=f_a(x)f_b^{-1}(x)$ is symmetric and relatively prime to
$x^n-1$ by Lemma \ref{mult}. From Theorem \ref{oddfenjie}, there
is a solution $g(x)=\sum\limits_{0\leq i\leq n-1}c_ix^i$ such that
$h(x)=g(x) g^*(x)$. It can be easily verified that
$(g(x),x^n-1)=1$.

Now let $\alpha=\sum\limits_{0\leq i\leq n-1}c_i\beta^{2^i}$.
Theorem \ref{guodu} shows that $\alpha$ is a normal element of
$\mathbb{F}_{2^{n}}$ over $\mathbb{F}_2$. Furthermore
$h(x)=g(x) g^{-1}(x)$ means that $f_a(x)=f_b(x)
g(x) g^*(x)$. By Lemma \ref{polybasis}, we know that the
corresponding vector of the normal element $\alpha$ is the given
vector $(a_0,a_1,\cdots,a_{n-1})$.
\end{proof}

Finally for $n=2^s\geq 4$ and odd $n$, we give the following algorithm to find a normal
element of $\mathbb{F}_{2^{n}}$ over $\mathbb{F}_2$ corresponding
to a given vector $(a_0,a_1,\cdots,a_{n-1})\in\mathbb{F}_2^n$.

\begin{algorithm}\label{oddalg}

\

Input: $(a_0,a_1,\cdots,a_{n-1})\in\mathbb{F}_2^n$, where $n=2^s\geq 4$ or $n$ is
odd.

Output: A normal element of $\mathbb{F}_{2^{n}}$ over
$\mathbb{F}_2$ such that
$\mbox{\textup{Tr}}_{2^n|2}(\alpha^{1+2^i})=a_i$ for all $0\leq
i\leq n-1$.

\vspace{3mm}

Step 1a: For $n=2^s\geq 4$, check whether $(a_0,a_1,\cdots,a_{n-1})\in\mathbb{F}_2^n$
satisfies $a_0=1$, $a_{n/2}=0$, $a_i=a_{n-i}$ for $1\leq i\leq
n/2-1$, and $\sum\limits_{1\leq i\leq n/2-1,(i,2)=1}a_i=1$. If
not, then output "There isn't such a normal element". Let
$f_a(x)=\sum\limits_{0\leq i\leq n-1}a_ix^i$.

Step 1b: For odd $n$, check whether $(a_0,a_1,\cdots,a_{n-1})\in\mathbb{F}_2^n$
is symmetric and $(f_a(x),\ x^n-1)=1$, where
$f_a(x)=\sum\limits_{0\leq i\leq n-1}a_ix^i$. If not, then output
"There isn't such a normal element".

Step 2: Find a normal element $\beta$ of $\mathbb{F}_{2^{n}}$ over
$\mathbb{F}_2$ (for example using the method in \cite{poli}).

Step 3: Compute $(b_0,b_1,\cdots,b_{n-1})$, where
$b_i=\mbox{\textup{Tr}}_{2^n|2}(\beta^{1+2^i})$.

Step 4: Use the standard extended GCD algorithm to compute
$f_b^{-1}(x)(\mbox{\textup{mod}}\ x^n-1)$, where
$f_b(x)=\sum\limits_{0\leq i\leq n-1}b_ix^i$.

Step 5a: For $n=2^s\geq 4$, solve the linear system (\ref{zu3}) to find the unique
solution $g(x)=\sum\limits_{0\leq i\leq n-1}c_ix^i\in G$ of
$h(x)=f_a(x)f_b^{-1}(x)=g(x)g^*(x)$.

Step 5b: For odd $n$, compute $h(x)=f_a(x)f_b^{-1}(x)=\sum\limits_{0\leq i\leq
n-1}h_ix^i\in\mathbb{F}_2[x]/(x^n-1)$. Let
$g(x)=\sum\limits_{0\leq i\leq
n-1}c_ix^i\in\mathbb{F}_2[x]/(x^n-1)$, where
$c_i=h_{2i(\mbox{\textup{mod}}\ n)}$.

Step 6: Output $\alpha=\sum\limits_{0\leq i\leq
n-1}c_i\beta^{2^i}$.

\end{algorithm}

\begin{remark}
The correctness of the above algorithm can be showed by Theorem \ref{2powercase} and Theorem \ref{oddcase}.  The running time of solving linear equation system in Step 5a is $O(n^3)$ $\mathbb{F}_2-$operations. Thus the expected computation complexity of the above algorithm can be bounded by $O(n^3)$.
\end{remark}

\begin{example}

Let $n=16$ and the irreducible polynomial of order 16 be $f(x)=x^{16}+x^5+x^3+x^2+1\in \mathbb{F}_{2}[x].$
Define the finite field $\mathbb{F}_{2^{16}}$ to be $\mathbb{F}_{2}[x]/(f(x))$. Suppose $\gamma\in \mathbb{F}_{2^{16}}$ is a root of $f(x)$.
We want to find a normal element $\alpha$ with corresponding vector $$a=(1,1,0,0,0,0,0,0,0,0,0,0,0,0,0,1)\in\mathbb{F}_2^{16}$$ by the above algorithm.

Step 1a: Check whether the vector $a$
satisfies $a_0=1$, $a_{8}=0$, $a_i=a_{16-i}$ for $1\leq i\leq
7$, and $\sum\limits_{1\leq i\leq 7,(i,2)=1}a_i=1$. Let
$f_a(x)=\sum\limits_{0\leq i\leq 15}a_ix^i=1+x+x^{15}$.

Step 2: Choose the normal element $\beta=\gamma+\gamma^{126}$ of $\mathbb{F}_{2^{16}}$ over
$\mathbb{F}_2$.

Step 3: Compute $b=(b_0,b_1,\cdots,b_{15})=(1,0,1,1,1,0,0,0,0,0,0,0,1,1,1,0)$, where
$b_i=\mbox{\textup{Tr}}_{2^{16}|2}(\beta^{1+2^i})$ for $0\leq i\leq 15$.

Step 4: Let $f_b(x)=\sum\limits_{0\leq i\leq 15}b_ix^i=1+x^2+x^3+x^4+x^{12}+x^{13}+x^{14}$. Use the standard extended GCD algorithm to compute
$f_b^{-1}(x)(\mbox{\textup{mod}}\ x^{16}-1)=1+ x+ x^2+ x^3+ x^5+ x^6+ x^{10}+ x^{11}+ x^{13}+ x^{14} +x^{15}$.

Step 5a: Compute $h(x)=f_a(x)f_b^{-1}(x)\pmod{x^{16}-1}=1+ x+ x^2+ x^7+ x^9+ x^{14}+ x^{15}$. By (\ref{zu3}) solve the linear equation system
\begin{equation*}
\left\{
\begin{array}{ll}
y_0&=h_0(z_0,\cdots,z_{15})=1,\\
y_1&=h_1(z_0,\cdots,z_{15})=z_{7}+z_1=1,\\
y_2&=h_2(z_0,\cdots,z_{15})=z_1=1,\\
y_3&=h_3(z_0,\cdots,z_{15})=z_1+z_3+z_{6}=0,\\
y_4&=h_4(z_0,\cdots,z_{15})=z_4+z_3=0,\\
y_5&=h_5(z_0,\cdots,z_{15})=z_5+z_4+z_{5}=0,\\
y_6&=h_6(z_0,\cdots,z_{15})=z_6+z_5=0,\\
y_7&=h_7(z_0,\cdots,z_{15})=z_7+z_6+z_{4}+z_3=1,\\
\end{array}
\right.
\end{equation*}
we obtain $g(x)=\sum\limits_{0\leq i\leq 15}c_ix^i\in G=1+ x+ x^5+ x^6+ x^9+ x^{10}+ x^{14}$, the unique
solution of
$h(x)=f_a(x)f_b^{-1}(x)=g(x)g^*(x)$.

Step 6: Output the normal element $\alpha=\sum\limits_{0\leq i\leq
15}c_i\beta^{2^i}=\sum\limits_{0\leq i\leq
15}c_i(\gamma+\gamma^{126})^{2^i}$, where $$(c_0,\cdots,c_{15})=(1,1,0,0,0,1,1,0,0,1,1,0,0,0,1,0).$$
\end{example}

\section{The General Case}
In this section, we will give some necessary conditions for a
vector to be a corresponding vector of a normal element of
$\mathbb{F}_{2^n}$ over $\mathbb{F}_2$ for general $n$ with $4|n$. Firstly
recall a result about constructing a normal element over sub-field
from a normal element over a larger field.
\begin{theorem}\label{perlis}\cite{perlis}
Let $t$ and $v$ be any positive integers. If $\alpha$ is a normal
element of $\mathbb{F}_{q^{vt}}$ over $\mathbb{F}_q$, then
$\gamma=\mbox{Tr}_{q^{vt}|q^t}(\alpha)$ is a normal element of
$\mathbb{F}_{q^{t}}$ over $\mathbb{F}_q$.
\end{theorem}

\begin{proposition}
Let $n=2^sm$, $2^s\geq 4$, $m$ be odd. If symmetric vector
$(a_0,a_1,\cdots,a_{n-1})\in\mathbb{F}_2^{n}$ corresponds to a
normal element of $\mathbb{F}_{2^{n}}$ over $\mathbb{F}_2$, then

1) $\sum\limits_{i=0}\limits^{m-1}a_{i2^s}=1$,
$\sum\limits_{i=0}\limits^{m-1}a_{i2^s+2^{s-1}}=0$,

2) $\sum\limits_{1\leq k\leq 2^{s-1}-1,
(k,2)=1}\sum\limits_{i=0}\limits^{m-1}a_{i2^s+k}=1$,

3) $(\sum\limits_{0\leq k\leq m-1
}(\sum\limits_{i=0}\limits^{2^s-1}a_{im+k})x^k, x^m-1)=1$.
\end{proposition}
\begin{proof}
Let $\alpha$ be the normal element of $\mathbb{F}_{2^{n}}$ over
$\mathbb{F}_2$ with corresponding vector
$(a_0,a_1,\cdots,a_{n-1})$. By Theorem \ref{perlis},
$\gamma=\mbox{Tr}_{2^n|2^{2^s}}(\alpha)$ is a normal element of
$\mathbb{F}_{2^{2^s}}$ over $\mathbb{F}_2$. Now we consider the
corresponding vector $(r_0,r_1,\cdots,r_{2^s-1})$ of $\gamma$.

For every $0\leq k\leq 2^s-1$,
\begin{eqnarray*}
  r_k&=&\mbox{Tr}_{2^{2^s}|2}(\gamma^{1+2^k})=\mbox{Tr}_{2^{2^s}|2}((\sum\limits_{i=0}\limits^{m-1}\alpha^{2^{i2^s}})^{1+2^k})
    = \mbox{Tr}_{2^{2^s}|2}(\sum\limits_{i=0}\limits^{m-1}\sum\limits_{j=0}\limits^{m-1}(\alpha\alpha^{2^{j2^s+k}})^{2^{i2^s}}) \\
  &=& \sum\limits_{0\leq l\leq
  2^s-1}(\sum\limits_{i=0}\limits^{m-1}\sum\limits_{j=0}\limits^{m-1}(\alpha\alpha^{2^{j2^s+k}})^{2^{i2^s}})^{2^l}
     = \sum\limits_{0\leq j\leq
  m-1}(\sum\limits_{v:=i2^s+l,v=0}\limits^{v=n-1}(\alpha\alpha^{2^{j2^s+k}})^{2^{v}})\\
   &=& \sum\limits_{0\leq j\leq m-1}\mbox{Tr}_{2^n|2}(\alpha\alpha^{2^{j2^s+k}})= \sum\limits_{0\leq j\leq m-1}a_{j2^s+k}.
\end{eqnarray*}

It can be proved that $(r_0,r_1,\cdots,r_{2^s-1})$ is symmetric
since $(a_0,a_1,\cdots,a_{2^s-1})$ is symmetric. Hence by
Theorem \ref{2powercase}, the necessary conditions for
$(r_0,r_1,\cdots,r_{2^s-1})$ to be a vector corresponding to a
normal element of $\mathbb{F}_{2^{2^s}}$ over $\mathbb{F}_2$ are:
$$\sum\limits_{i=0}\limits^{m-1}a_{i2^s}=1,\ \sum\limits_{i=0}\limits^{m-1}a_{i2^s+2^{s-1}}=0,\ \mbox{and}\
\sum\limits_{1\leq k\leq 2^{s-1}-1,
(k,2)=1}\sum\limits_{i=0}\limits^{m-1}a_{i2^s+k}=1.$$

Similarly, by Theorem \ref{perlis} $(t_0,t_1,\cdots,t_{m-1})$ where
$t_k=\sum\limits_{0\leq j\leq m-1}a_{j2^s+k}$ for $0\leq k\leq
m-1$ can be showed to be the corresponding vector of the normal
element $\beta=\mbox{Tr}_{2^{n}|2^m}(\alpha)$ of
$\mathbb{F}_{2^m}$ over $\mathbb{F}_2$. So from Theorem \ref{oddcase},
we have $(\sum\limits_{0\leq k\leq m-1
}(\sum\limits_{i=0}\limits^{2^s-1}a_{im+k})x^k, x^m-1)=1$.
\end{proof}

\begin{theorem}\label{pincin}\cite{pincin}
Let $n=tv$ with $t$ and $v$ relatively prime. Then for $\alpha\in
\mathbb{F}_{q^v}$ and $\beta\in\mathbb{F}_{q^t}$, the element
$\gamma=\alpha\beta$ is a normal element of $\mathbb{F}_{q^{vt}}$
over $\mathbb{F}_q$, if and only if $\alpha$ is a normal element of
$\mathbb{F}_{q^{v}}$ over $\mathbb{F}_q$ and $\beta$ is a normal
element of $\mathbb{F}_{q^{t}}$ over $\mathbb{F}_q$.
\end{theorem}

By the above theorem, we have
\begin{proposition}\label{generalconst}
Suppose symmetric vector
$(a_0,a_1,\cdots,a_{2^s-1})\in\mathbb{F}_2^{2^s}$ satisfies
$a_0=1$, $a_{n/2}=0$ and $\sum\limits_{0\leq i\leq
2^{s-1}-1,(i,2)=1}a_i=1$, symmetric vector
$(b_0,b_1,\cdots,b_{m-1})\in\mathbb{F}_2^{m}$ satisfies
$(\sum\limits_{0\leq i\leq m-1}b_ix^i, x^m-1)=1$, where $m$ is odd
and $2^s\geq 4$. Then there exists a normal element of
$\mathbb{F}_{2^{2^sm}}$ over $\mathbb{F}_2$ with corresponding
vector $(c_0,c_1,\cdots,c_{2^sm-1})\in\mathbb{F}_2^{2^sm}$ such
that $c_k=a_{k(\mbox{\textup{mod}}\ 2^s)}b_{k(\mbox{\textup{mod}}\
m)}$ for $0\leq k\leq 2^sm-1$.
\end{proposition}
\begin{proof}
By the above assumptions and Algorithm
\ref{oddalg}, we can construct a normal element $\alpha$ of
$\mathbb{F}_{2^{2^s}}$ over $\mathbb{F}_2$ with corresponding
vector $(a_0,\cdots,a_{2^s-1})$ and a normal element $\beta$ of
$\mathbb{F}_{2^{m}}$ over $\mathbb{F}_2$ with corresponding vector
$(b_0,\cdots,b_{m-1})$. From Theorem \ref{pincin},
$\gamma=\alpha\beta$ is a normal element of
$\mathbb{F}_{2^{2^sm}}$ over $\mathbb{F}_2$.

For $0\leq k\leq 2^sm-1$, $\alpha^{2^k}=\alpha^{2^{k(\mbox{mod}\
2^s)}}$ and $\beta^{2^k}=\beta^{2^{k(\mbox{mod}\ m)}}$, thus
\begin{eqnarray*}
  \mbox{Tr}_{2^{2^sm}|2}(\gamma^{1+2^k}) &=& \mbox{Tr}_{2^{2^sm}|2}(\alpha\beta\alpha^{2^k}\beta^{2^k}) \\
   &=& \mbox{Tr}_{2^{2^sm}|2}(\alpha\beta\alpha^{2^{k(\mbox{mod}\ 2^s)}}\beta^{2^{k(\mbox{mod}\ m)}}) \\
   &=& \mbox{Tr}_{2^{2^s}|2}(\alpha \alpha^{2^{k(\mbox{mod}\ 2^s)}}) \mbox{Tr}_{2^m|2}(\beta\beta^{2^{k(\mbox{mod}\ m)}}) \\
   &=& a_{k(\mbox{mod}\ 2^s)}b_{k(\mbox{mod}\ m)}\\
   &=& c_k.
\end{eqnarray*}
So the corresponding vector of normal element $\gamma$ is
$(c_0,c_1,\cdots,c_{2^sm-1})$.
\end{proof}

Let $n=2^sm$ with $2^s\geq 4$ and $m$ odd,
$(a_0,a_1,\cdots,a_{2^s-1})\in\mathbb{F}_2^{2^s}$ be
\begin{equation*}
a_l=\left\{
\begin{array}{ll}
&1,\qquad l=0, i_0, 2^s-i_0, \\
&0,\qquad Otherwise,
\end{array}\right.
\end{equation*}
where $2\nmid i_0, i_0\in [1,2^s-1]$, and
$(b_0,b_1,\cdots,b_{m-1})=(1,0,\cdots,0)\in\mathbb{F}_2^{m}$.
Then the vector
$(c_0,c_1,\cdots,c_{2^sm-1})\in\mathbb{F}_2^{2^sm}$ defined by
$c_k=a_{k(\mbox{mod}\ 2^s)}b_{k(\mbox{mod}\ m)}$ for $0\leq k\leq
2^sm-1$ is
\begin{equation*}
c_k=\left\{
\begin{array}{ll}
&1,\qquad k=0, j_0m, n-j_0m, \\
&0,\qquad Otherwise,
\end{array}\right.
\end{equation*}
where $j_0$ is the unique solution of the congruence equation
$mx\equiv i_0(\mbox{mod}\ 2^s)$. By Corollary \ref{generalconst},
there exists a normal element of $\mathbb{F}_{2^{n}}$ over
$\mathbb{F}_2$ with its corresponding vector
$(c_0,\cdots,c_{n-1})$, the Hamming weight of which is $3$.

\begin{corollary}
For every $n$ with 4 dividing $n$, there is a normal element $\alpha$ of
$\mathbb{F}_{2^{n}}$ over $\mathbb{F}_2$ such that the Hamming
weight of its corresponding vector
$$(\textup{\mbox{Tr}}_{2^n|2}(\alpha),
\textup{\mbox{Tr}}_{2^n|2}(\alpha^{1+2}),
\textup{\mbox{Tr}}_{2^n|2}(\alpha^{1+4}), \cdots,
\textup{\mbox{Tr}}_{2^n|2}(\alpha^{1+2^{n-1}}))$$ is $3$.
\end{corollary}

\end{document}